\documentclass{article}

\usepackage{arxiv}

\usepackage[utf8]{inputenc} 
\usepackage[T1]{fontenc}    
\usepackage[hypertexnames=true]{hyperref}       
\usepackage{url}            
\usepackage{booktabs}       
\usepackage{amsmath, amsfonts}       
\usepackage{amsthm}
\usepackage{nicefrac}       
\usepackage{microtype}      
\usepackage{cleveref}       
\usepackage{graphicx}
\usepackage{doi}

\usepackage{amssymb}
\usepackage{algorithmic}
\usepackage[ruled,linesnumbered,vlined]{algorithm2e}
\usepackage{array}
\usepackage[caption=false,font=normalsize,labelfont=sf,textfont=sf]{subfig}
\usepackage{textcomp}
\usepackage{stfloats}
\usepackage{verbatim}
\usepackage{cite}
\usepackage{multirow} 
\usepackage{scalerel}
\usepackage{tikz}
\usepackage[title,titletoc]{appendix}

\DeclareMathOperator*{\argmin}{\arg\min}

\theoremstyle{plain}

\newtheorem{lemma}{Lemma}
\newtheorem{proposition}{Proposition}
\theoremstyle{definition}

\newtheorem{definition}{Definition}
\theoremstyle{remark}
\newtheorem{remark}{Remark}

\title{Improved Front Steepest Descent \\ for Multi-objective Optimization}

\author{ \href{https://orcid.org/0000-0002-2488-5486}{\includegraphics[scale=0.06]{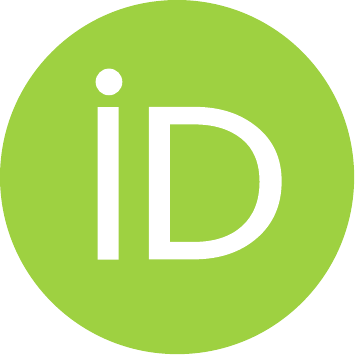}\hspace{1mm}Matteo Lapucci}\\
	Global Optimization Laboratory (GOL) \\
	Department of Information Engineering \\
	University of Florence \\
	Via di Santa Marta, 3, 50139, Florence, Italy \\
	\texttt{matteo.lapucci@unifi.it} \\
	\And
	\href{https://orcid.org/0000-0002-1394-0937}{\includegraphics[scale=0.06]{orcid.pdf}\hspace{1mm}Pierluigi Mansueto} \\
	Global Optimization Laboratory (GOL) \\
	Department of Information Engineering \\
	University of Florence \\
	Via di Santa Marta, 3, 50139, Florence, Italy \\
	\texttt{pierluigi.mansueto@unifi.it} \\
}

\renewcommand{\headeright}{Matteo Lapucci and Pierluigi Mansueto}
\renewcommand{\undertitle}{}
\renewcommand{\shorttitle}{Improved Front Steepest Descent for MOO}

\hypersetup{
pdftitle={Improved Front Steepest Descent for Multi-objective Optimization},
pdfsubject={},
pdfauthor={Matteo Lapucci, Pierluigi Mansueto},
pdfkeywords={Multi-objective optimization, Steepest descent, Pareto front},
}

\begin{document}
\maketitle

\begin{abstract}
	In this paper, we deal with the Front Steepest Descent algorithm for multi-objective optimization. We point out that the algorithm from the literature is often incapable, by design, of spanning large portions of the Pareto front. We thus introduce some modifications within the algorithm aimed to overcome this significant limitation. We prove that the asymptotic convergence properties of the algorithm are preserved and numerically show that the proposed method significantly outperforms the original one.
\end{abstract}


\MSCs{90C29 \and 90C30}

\section{Introduction} 
\label{sec:intro}
In this paper, we are interested in optimization problems of the form 
\begin{equation}
	\label{eq:mo_prob}
	\min_{x\in\mathbb{R}^n}\; F(x) = (f_1(x),\ldots,f_m(x))^T,
\end{equation} where $F:\mathbb{R}^n\to \mathbb{R}^m$ is a vector-valued continuously differentiable function. We are thus dealing with \textit{smooth, unconstrained multi-objective optimization} problems, where many functions have to be simultaneously minimized and Pareto's efficiency concepts have to be considered to establish optimality. We refer the reader to \cite{eichfelder2021twenty} for an introduction to multi-objective optimization.

\textit{Multi-objective descent methods} \cite{fliege2009newton,fliege2000steepest,gonccalves2020globally,Tanabe2019} constitute a class of algorithmic approaches designed to tackle these problems; these approaches basically extend classical iterative optimization algorithms for scalar optimization to the multi-objective setting. Descent methods are receiving increasing attention and have consistently become significant alternatives to scalarization methods \cite{drummond2008choice,eichfelder2009adaptive,pascoletti1984scalarizing} and evolutionary algorithms \cite{deb2002fast}. This is particularly true for recent versions of descent approaches that are specifically designed to handle sets of points and to construct an approximation of the entire \textit{Pareto front}, rather than a single solution.

In this short manuscript, we focus on the \textit{Front Steepest Descent} (\texttt{FSD}) algorithm proposed in \cite{cocchi2020convergence}. In particular, we argue that, although being far superior than the original single point steepest descent algorithm for multi-objective optimization \cite{fliege2000steepest}, \texttt{FSD} as defined in \cite{cocchi2020convergence} has limited exploration capabilities and it is quite frequently unable to span large portions of the Pareto front.

We thus propose small but crucial modifications to the algorithm, that allow to turn it tremendously effective at spanning the entire Pareto front, regardless of the starting set of points. We show that the proposed approach still enjoys the nice convergence guarantees of the original \texttt{FSD}. 

The rest of the paper is organized as follows: in Section \ref{sec:steepest}, we summarize the \texttt{FSD} algorithm, recalling its convergence properties; we then point out in Section \ref{sec:no_span} that in certain, common situations the algorithm is unable to span large portions of the Pareto front. In Section \ref{sec:improved_fsd} we introduce the novel strategy for generating nondominated solutions within \texttt{FSD} and we provide the convergence analysis for the resulting algorithm in Section \ref{sec:convergence}. In Section \ref{sec:exp}, we present the results of numerical experiments showing that the proposed modification significantly improves effectiveness and consistency of the \texttt{FSD} algorithm. We finally give some concluding remarks in Section \ref{sec:conclusions}.

\section{The Front Steepest Descent algorithm}
\label{sec:steepest}

The Front Steepest Descent algorithm \cite{cocchi2020convergence} was designed to solve problem \eqref{eq:mo_prob} according to Pareto's optimality concepts. Given the standard partial ordering in $\mathbb{R}^m$, i.e.,
\begin{gather*}
	u\le v \iff u_j\le v_j, \;\forall\,j=1,\ldots,m,\\
	u< v \iff u_j< v_j, \;\forall\,j=1,\ldots,m,\\
	u \lneqq v\iff u\le v \land u\neq v,
\end{gather*}
the aim is to find solutions $\bar{x}\in\mathbb{R}^n$ that satisfy the following properties, listed in decreasing order of strength:
\begin{itemize}
	\item \textit{Pareto optimality:} $\nexists\,y\in\mathbb{R}^n$ s.t. $F(y)\lneqq F(\bar{x})$; 
	\item \textit{Weak Pareto optimality:} $\nexists\,y\in\mathbb{R}^n$ s.t. $F(y)< F(\bar{x})$;  
	\item \textit{Pareto stationarity:} $\min\limits_{d\in\mathbb{R}^n}\max\limits_{j=1,\ldots,m}\nabla f_j(\bar{x})^Td = 0$. 
\end{itemize}
In fact, there typically exist many Pareto optimal solutions (the \textit{Pareto set}) that account for different trade-offs between the contrasting objectives; these trade-offs, that constitute in the objectives space the \textit{Pareto front}, can a posteriori be evaluated by the decision makers, who are thus willing to have the broadest possible range of available options.  

\texttt{FSD} method specifically aims to construct an approximation of the entire Pareto front; the algorithm works in an iterative fashion, maintaining at each iteration a set $X^k$ of solutions that are \textit{mutually nondominated}, i.e., for any $x\in X^k$ there is no $y\in X^k$ such that $F(y)\lneqq F(x)$.

The points for the set $X^{k+1}$ are computed carrying out search steps starting from the points $\hat{x} \in X^k$ along:
\begin{itemize}
	\item the \textit{steepest common descent direction} \cite{fliege2000steepest}: 
	\begin{equation}
		\label{eq::steepest_common}
		v(\hat{x}) = \argmin\limits_{d\in\mathbb{R}^n}\max\limits_{j=1,\ldots,m} \nabla f_j(\hat{x})^Td + \frac{1}{2}\|d\|^2;
	\end{equation}
	\item the \textit{steepest partial descent directions} \cite{cocchi2021pareto,cocchi2020convergence}: given $I\subset \{1,\ldots,m\}$, 
	\begin{equation}
		\label{eq::steepest_partial}
		v^I(\hat{x}) = \argmin\limits_{d\in\mathbb{R}^n}\;\max\limits_{j\in I}\;\nabla f_j(\hat{x})^Td + \frac{1}{2}\|d\|^2.
	\end{equation}
\end{itemize}
The use of equality notation in the definition of steepest descent directions is justified by the uniqueness of the solution set for the above optimization problems (the objective is strongly convex and continuous). Given any subset of objectives $I$, a partial descent direction exists if $$\theta^I(\hat{x}) = \min\limits_{d\in\mathbb{R}^n}\;\max\limits_{j\in I}\;\nabla f_j(\hat{x})^Td + \frac{1}{2}\|d\|^2<0;$$  of course, the steepest common descent direction $v(\hat{x})$ and the corresponding $\theta\left(\hat{x}\right)$ are considered when $I=\{1,\ldots,m\}$. Both mappings $v^I(\hat{x})$ and $\theta^I(\hat{x})$ are continuous \cite{fliege2000steepest}.

The instructions of the \texttt{FSD} procedure are summarized in Algorithm \ref{alg::FSD}. In brief, at each iteration $k$, all points in the current set of nondominated solutions, $X^k$, are considered; for each one of these points, $x_c$, a line search along the steepest partial descent direction is carried out for any subset of objectives $I\subseteq\{1,\ldots,m\}$ such that $\theta^I(x_c)<0;$ in addition, a subset $I$ is only considered for $x_c$ if the point is nondominated with respect to that subset of objectives. 

\SetInd{1ex}{1ex}
\begin{algorithm}[h]
	\caption{\texttt{FrontSteepestDescent}} \label{alg::FSD}
	Input: $F:\mathbb{R}^n \rightarrow \mathbb{R}^m$, $X^0$ set of mutually nondominated points w.r.t.\ $F$. \\
	$k = 0$\\
	\While{a stopping criterion is not satisfied}{
		$\hat{X}^k = X^k$ \\
		\ForAll{$x_c\in X^k$}{
			\ForAll{$I\subseteq\{1,\ldots,m\}$ such that \label{step:forall} \begin{itemize}
					\item $\nexists y\in \hat{X}^k \text{ s.t. } F_I(y)\lneqq F_I(x_c)$ and
					\item $\theta^I(x_c) < 0$
			\end{itemize}}{
				$\alpha$ = \texttt{ArmijoLS}($F(\cdot), I, \hat{X}^k, x_c, v^I(x_c), \theta^I(x_c)$) \\
				$\hat{X}^k = \hat{X}^k \setminus \{y \in \hat{X}^k \mid F(x_c + \alpha v^I(x_c)) \lneqq F(y)\} \cup \{x_c + \alpha v^I(x_c)\}$\label{line::point-insert}
			}			
		}
		$X^{k + 1} = \hat{X}^k$ \\
		$k = k + 1$
	}
	\Return $X^k$
\end{algorithm}

The line search is an Armijo-type procedure whose scheme is reported in Algorithm \ref{alg::als}. Given a nondominated point and a search direction w.r.t.\ the objectives in $I$, the algorithm returns a new point such that it is ``sufficiently nondominated''. The obtained point is added to the set of nondominated points, while all the points that are now dominated by it are filtered out.

\begin{algorithm}[h]
	\caption{\texttt{ArmijoLS}} \label{alg::als}
	Input: $F:\mathbb{R}^n \rightarrow \mathbb{R}^m$, $I\subseteq\{1,\ldots,m\}$, $\hat{X}$ set of mutually nondominated points w.r.t.\ $F$, $x_c\in \hat{X}$, $v^I(x_c) \in \mathbb{R}^n$, $\theta^I(x_c) \in \mathbb{R}$, $\alpha_0>0$, $\delta\in(0,1)$, $\gamma\in(0,1)$. \\
	$\alpha = \alpha_0$\\
	Let $\hat{X}_I$ be the set of points in $\hat{X}$ that are mutually nondominated w.r.t.\ $F_I$\\
	\While{$\exists\, y\in \hat{X}_I$ s.t.\ $F_I(y) + \boldsymbol{1}\gamma\alpha\theta^I(x_c) < F_I(x_c+\alpha v^I(x_c))$}{
		$\alpha= \delta\alpha$	
	}
	\Return $\alpha$
\end{algorithm}

Algorithm \ref{alg::als} enjoys the following finite termination properties. 

\begin{proposition}[{\cite[Proposition 4]{cocchi2020convergence}}]
	Let $I\subseteq\{1,\ldots,m\}$, $\hat{X}$ be a set of mutually nondominated solutions containing $x_c$; $x_c$ is also nondominated w.r.t.\ $F_I$ and it is such that $\theta^I(x_c)<0$. Then, $\exists\,\alpha>0$, sufficiently small, such that
	$$ F_I(y) + \boldsymbol{1}\gamma\alpha\theta^I(x_c) \nless F_I(x_c+\alpha v^I(x_c)), \quad\forall\,y\in\hat{X}_I,$$
	with $\hat{X}_I$ being the set of points in $\hat{X}$ that are mutually nondominated w.r.t.\ $F_I$. Furthermore, the produced point $x_c+\alpha v^I(x_c)$ is nondominated by any point in $\hat{X}$ with respect to $F$.
\end{proposition}

\begin{remark}
	An improved version of Algorithm \ref{alg::als} was also proposed in \cite{cocchi2020convergence}, which is based on an extrapolation strategy and allows to possibly obtain many nondominated solutions along the search direction. When used within Algorithm \ref{alg::FSD}, the extrapolation technique does not alter theoretical convergence results, but the resulting algorithm is reported to be significantly more effective.
\end{remark}

Now, we shall recall the convergence properties of Algorithm \ref{alg::FSD}, which are based on the concept of \textit{linked sequence} \cite{liuzzi2016derivative}.

\begin{definition}
	Let $\{X^k\}$ be the sequence of sets of nondominated points produced
	by Algorithm \ref{alg::FSD}. We define a linked sequence as a sequence $\{x_{j_k}\}$ such that, for any $k=1,2,\ldots$, the point $x_{j_k}\in X^k$ is generated at iteration $k-1$ of Algorithm \ref{alg::FSD} by the point $x_{j_{k-1}}\in X^{k-1}$.
\end{definition}

\begin{proposition}[{\cite[Proposition 5]{cocchi2020convergence}}]
	Let us assume that there exists $x_0\in X^0$ such that
	\begin{itemize}
		\item $x_0$ is not Pareto stationary;
		\item the set $\mathcal{L}(x_0) = \bigcup_{j=1}^{m}\{x\in\mathbb{R}^n\mid f_j(x)\le f_j(x_0)\}$ is compact.
	\end{itemize}
	Let $\{X^k\}$ be the sequence of sets of nondominated points produced by Algorithm \ref{alg::FSD}. Let $\{x_{j_k}\}$ be a linked sequence, then it admits limit points and every limit point is Pareto-stationary for problem \eqref{eq:mo_prob}.
\end{proposition}

\begin{figure*}
	\centering
	\subfloat[\label{fig::FIG_FSDA-a}]{\includegraphics[width=0.35\textwidth]{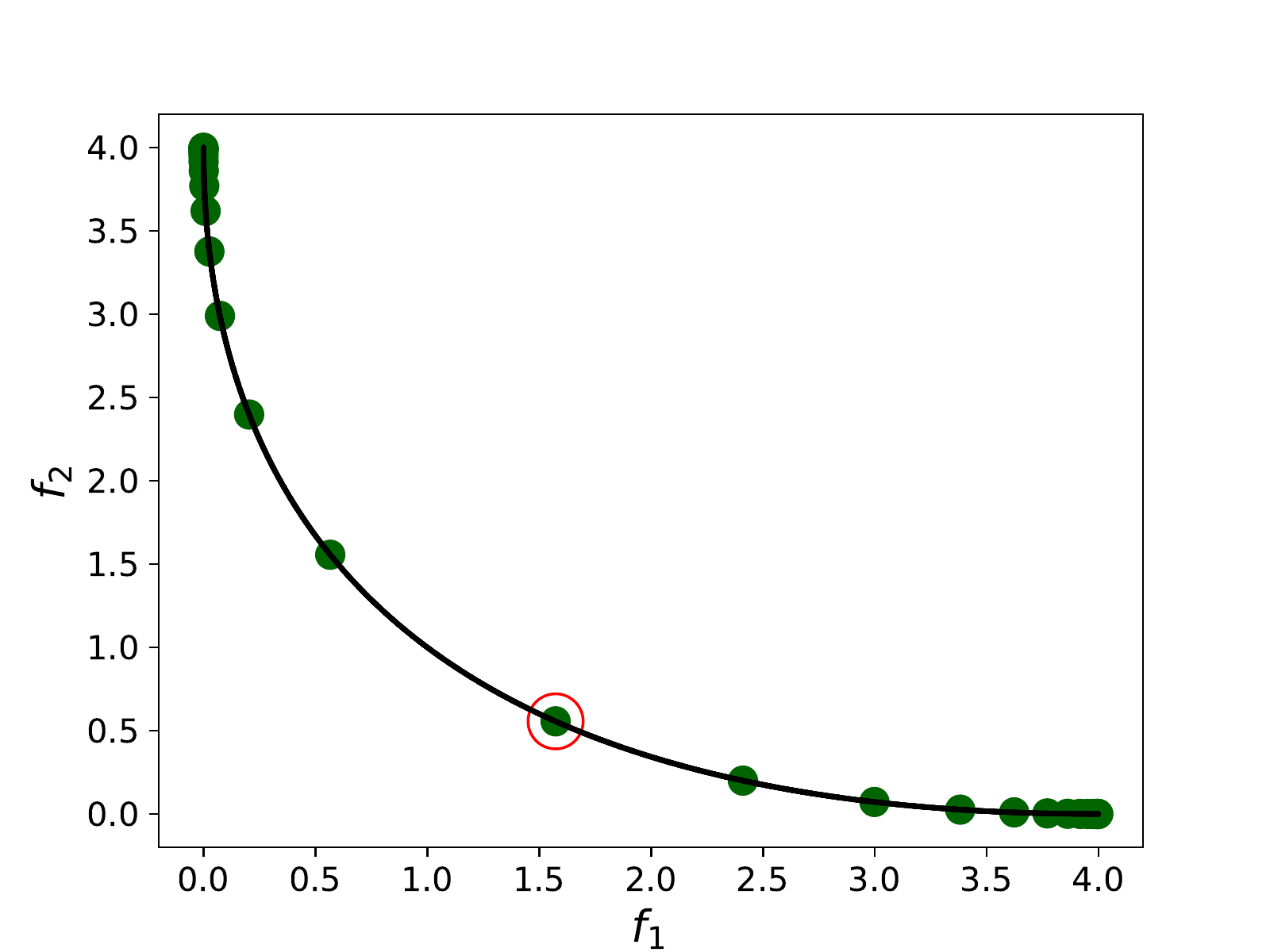}}
	\hfil
	\subfloat[\label{fig::FIG_FSDA-b}]{\includegraphics[width=0.35\textwidth]{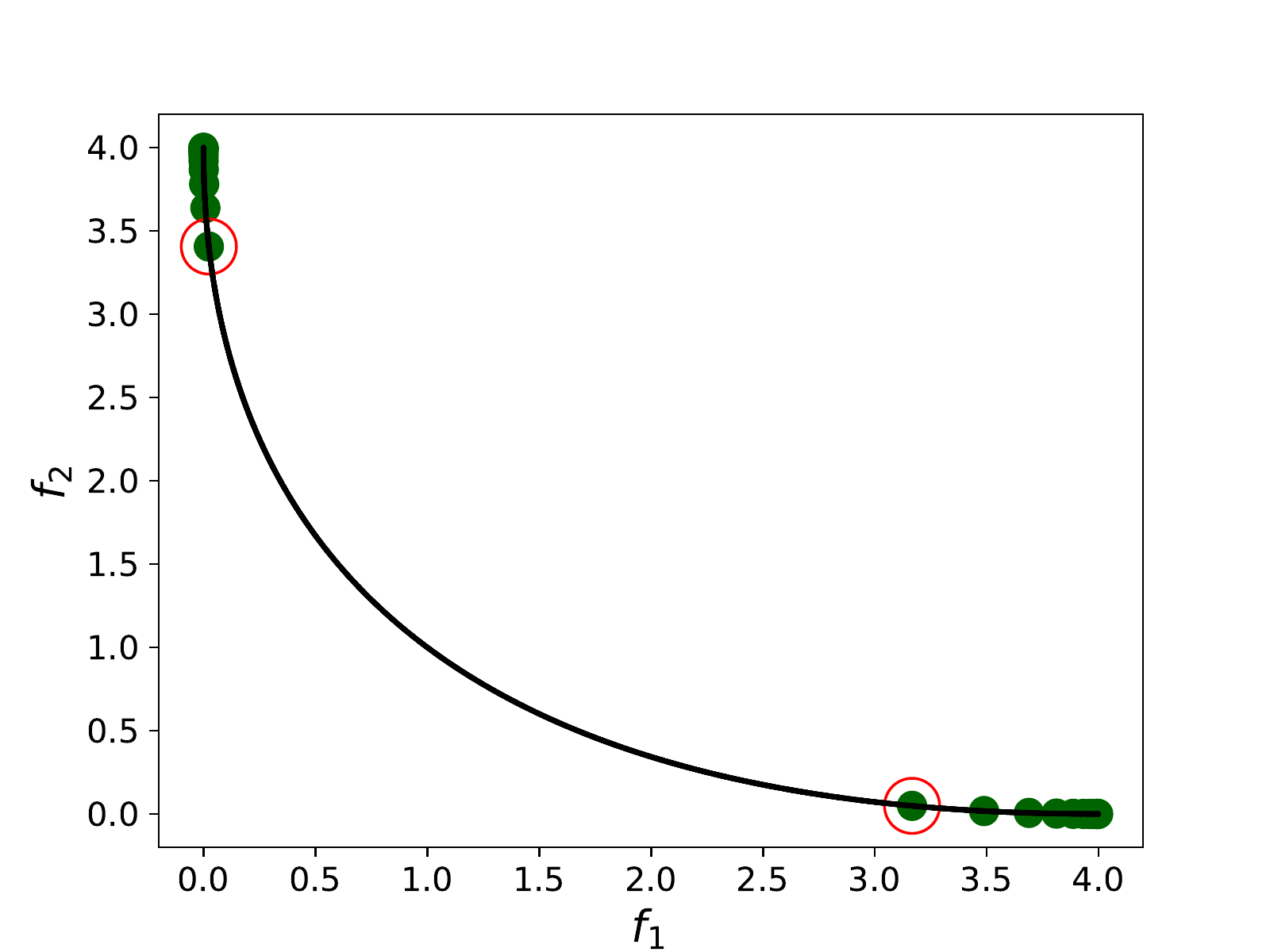}}
	\\
	\subfloat[\label{fig::FIG_FSDA-c}]{\includegraphics[width=0.35\textwidth]{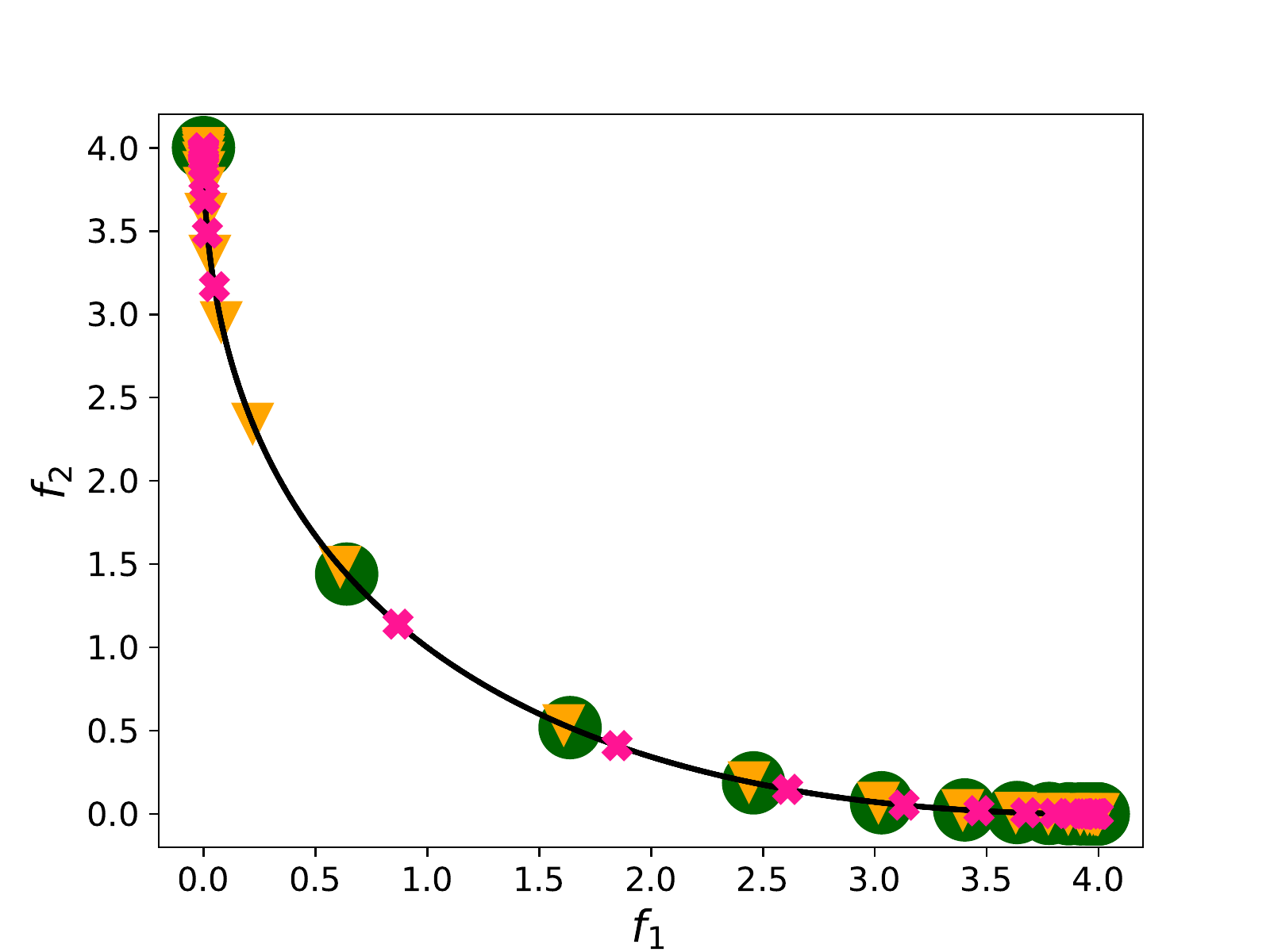}}
	\hfil
	\subfloat[\label{fig::FIG_FSDA-d}]{\includegraphics[width=0.35\textwidth]{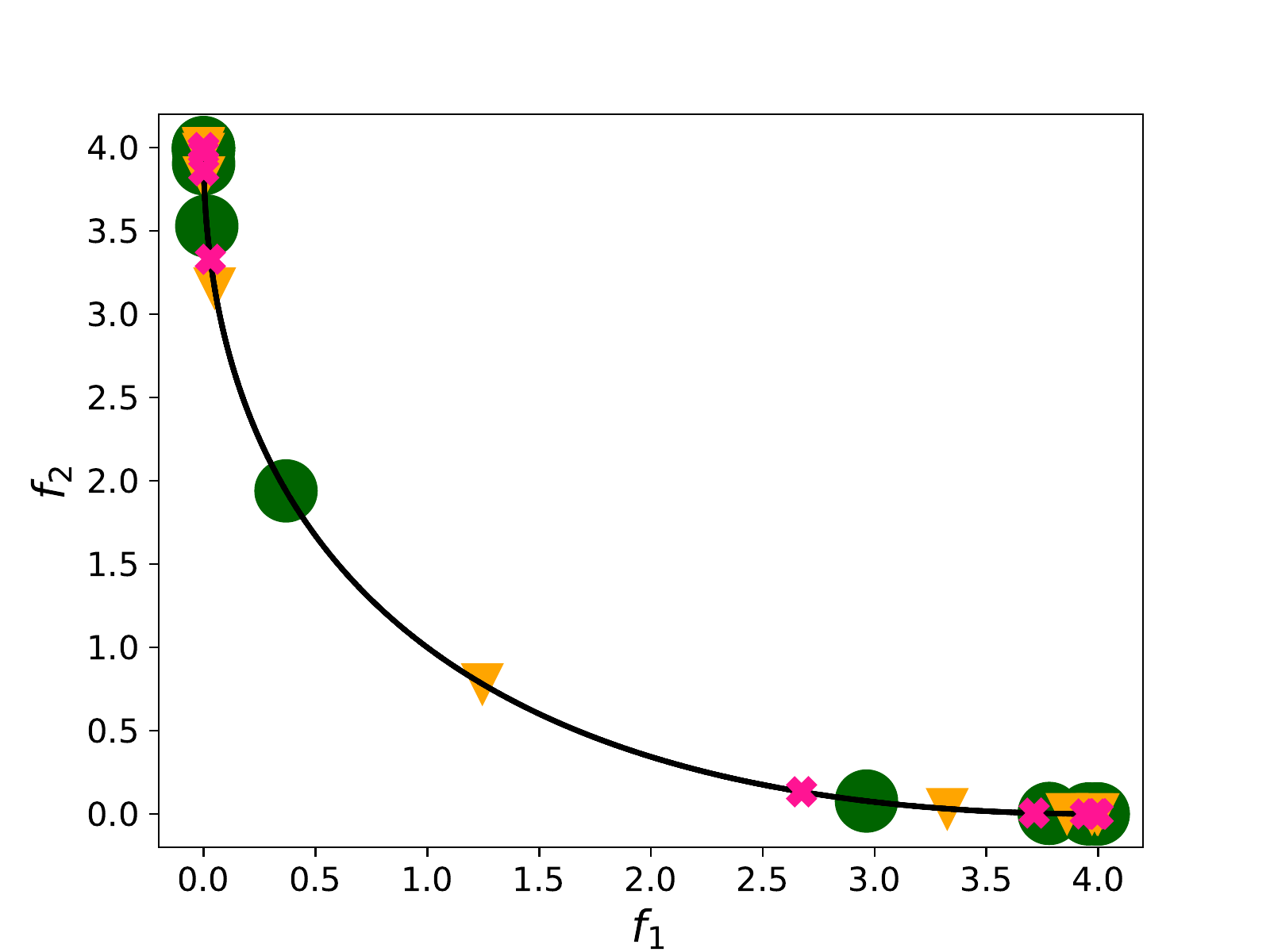}}
	\caption{Pareto fronts obtained by the \texttt{FSD} algorithm on the convex JOS problem ($n=5$). (a) \texttt{FSD} starts from 1 Pareto point; (b) \texttt{FSD} starts from 2 Pareto points; (c) 3 independent \texttt{FSD} runs, started from 3 different random points; (d) 3 independent runs of \texttt{FSD} with the extrapolation strategy, started from the same 3 random points as in (c).}
	\label{fig::FIG_FSDA}
\end{figure*}

\subsection{\texttt{FSD} may not span the Pareto front}
\label{sec:no_span}
The \texttt{FSD} algorithm constitutes, in practice, a significant improvement w.r.t.\ the simple multi-start steepest descent strategy for multi-objective optimization. However, in experimental settings, it is not uncommon to observe situations where \texttt{FSD} is unable to retrieve large portions of the Pareto front.

Here, we highlight this shortcoming and argue that it is the direct result of algorithmic design. In particular, the first condition at step \ref{step:forall} of Algorithm \ref{alg::FSD} makes the outcome of the algorithm very strongly dependent on the starting point(s).

When a point $x_c$ is considered for exploration in Algorithm \ref{alg::FSD}, a partial descent direction obtained according to the subset of objectives $I\subseteq \{1,\ldots,m\}$ is only considered if $x_c$ is nondominated within $X^k$ w.r.t.\ $F_I$; in other words, there is no $y\in X^k$ such that $F_I(y)\lneqq F_I(x_c)$. This condition was required by the authors of \cite{cocchi2020convergence} in order to establish finite termination properties for the line search (Algorithm \ref{alg::als}).

Unfortunately, that same condition results in a limited fraction of points in $X^k$ to be used for starting a partial descent search. 
This fact can be visualized, with very extreme outcomes, in the bi-objective case; indeed, when $m=2$, for each of the two proper subsets of indices, $I_1=\{1\}$ and $I_2=\{2\}$ there is only one point that satisfies the (partial) nondominance condition: $x_{I_1} = \argmin_{x\in X^k}f_1(x)$ and $x_{I_2} = \argmin_{x\in X^k}f_2(x)$.

Thus, partial descent is only carried out starting from the two current extreme points in the Pareto front. Moreover, these partial descent steps will only allow to explore, outwards, the extreme parts of the current front approximation, whereas the other descent step will mainly drive points to Pareto stationarity; as a result, even large holes within the current solutions set cannot be filled.

Taking the reasoning to the extreme, let us assume that the starting set of solutions already lies on the Pareto front; if the set contains only one point, then by repeated partial descent w.r.t.\ $I_1$ and $I_2$ the entire Pareto front can be spanned quite uniformly; this situation is depicted in Figure \ref{fig::FIG_FSDA-a}. If, on the other hand, there are two starting solutions, possibly far away from each other in the objective space, then only the extreme parts of the front will be spanned, while the gap between the two points is not tackled (Figure \ref{fig::FIG_FSDA-b}). Of course, the same reasoning applies with more than two starting points.

The paradoxical behavior of the algorithm is such that it might be convenient to start far away from the Pareto front. In this way, \texttt{FSD} may have many iterations at its disposal to increase the size of the set $X^k$ and uniformly span the objectives space; points are then driven to Pareto stationarity thanks to steps carried out considering $I=\{1,2\}$. Anyhow, the results are still influenced, somewhat randomly, by the starting solutions, as shown in Figure \ref{fig::FIG_FSDA-c}. Moreover, the extreme parts of the front are always spanned much more densely than the central one. We shall remark that, as the intermediate regions of the front often provide the most interesting trade-offs to users, this is a very significant issue in practice.

The extrapolation technique proposed in \cite{cocchi2020convergence} might allow to partly alleviate the issue discussed here, as much more nondominated solutions are obtained at each iteration; however, it is again the exploration of the extreme regions that is mainly enhanced and sped up, with possibly overall counterproductive results (Figure \ref{fig::FIG_FSDA-d}).

\section{Improved Front Steepest Descent}
\label{sec:improved_fsd}
In Algorithm \ref{alg::improved_FSD}, we report the scheme of a modified Front Steepest Descent (\texttt{IFSD}) algorithm that overcomes the limitations of Algorithm \ref{alg::FSD} discussed in Section \ref{sec:no_span}. 

\SetInd{1ex}{1ex}
\begin{algorithm}[!h]
	\caption{\texttt{ImprovedFrontSteepestDescent}} \label{alg::improved_FSD}
	Input: $F:\mathbb{R}^n \rightarrow \mathbb{R}^m$, $X^0$ set of mutually nondominated points w.r.t.\ $F$, $\alpha_0>0,$ $\delta\in(0,1),\gamma\in(0,1)$. \\
	$k = 0$\\
	\While{a stopping criterion is not satisfied}{
		$\hat{X}^k = X^k$ \\
		\ForAll{$x_c\in X^k$}{
			\If{$x_c \in \hat{X}^k$ \label{step:x_c}}{
				\If{$\theta(x_c)<0$ \label{step.if_cond}}{
					$\alpha_c^k = \max_{h=0,1,\ldots} \{\alpha_0\delta^h\mid F(x_c+\alpha_0\delta^h v(x_c))\le F(x_c)+\boldsymbol{1}\gamma\alpha_0\delta^h\theta(x_c)\}$ \label{step:line_search}
				}
				\Else{$\alpha_c^k=0$ \label{step:end_if_else}}
				$z_c^k = x_c+\alpha_c^kv(x_c)$ \label{step:z}\\
				$\hat{X}^k = \hat{X}^k \setminus \{y \in \hat{X}^k \mid F(z^k_c) \lneqq F(y)\} \cup \{z^k_c\}$ \label{step:add_zk}\\
				\ForAll{$I\subseteq\{1,\ldots,m\}$ s.t.\ $\theta^I(z_c^k) < 0$}{
					\If{$z_c^k\in\hat{X}^k$}{
						$\alpha_c^I$ = $\max_{h=0,1,\ldots} \{\alpha_0\delta^h\mid \forall\,y\in\hat{X}^k\,\exists j \in \left\{1,\ldots, m\right\}\text{ s.t. } f_j(z_c^k+\alpha_0\delta^h v^I(z_c^k)) < f_j(y)\}$ \label{step:second_ls} \\
						$\hat{X}^k = \hat{X}^k \setminus \{y \in \hat{X}^k \mid F(z_c^k + \alpha_c^I v^I(z_c^k)) \lneqq F(y)\} \cup \{z_c^k + \alpha_c^I v^I(z_c^k)\}$ \label{step:insert_partial}\\
				}}
				
			}
			
		}
		$X^{k + 1} = \hat{X}^k$ \\
		$k = k + 1$
	}
	\Return $X^k$
\end{algorithm}

Algorithm \ref{alg::improved_FSD} includes a bunch of modifications w.r.t.\ the original \texttt{FSD} approach:
\begin{itemize}
	\item for any point in $X^k$ that is still nondominated when it is considered for exploration, a preliminary steepest descent step is carried out; this step exploits a classical single point Armijo line search \cite{fliege2000steepest};
	\item further searches w.r.t.\ subsets of objectives start at the obtained point, as long as it is not dominated;
	\item for partial descent searches, we require the obtained point to be nondominated by all other points in $\hat{X}^k$.
\end{itemize}

The idea is that, with these modifications, all points may be used to start exploration based on partial descent; convergence of all the produced points towards stationarity is then forced by means of the ``preliminary'' steepest descent step, that ensures the sufficient decrease. In the next section we prove that the algorithm is well defined and actually produces convergent sequences of points.

\subsection{Convergence analysis}
\label{sec:convergence}
In this section, we provide the formal convergence analysis for Algorithm \ref{alg::improved_FSD}.

\begin{proposition}
	The line search at step \ref{step:line_search} of Algorithm \ref{alg::improved_FSD} is well defined.
\end{proposition}
\begin{proof}
	The result follows from \cite[Lemma 4]{fliege2000steepest} and by the if condition at step \ref{step.if_cond} that ensures that $\theta(x_c)<0$.
\end{proof}

\begin{proposition}
	\label{prop:well_def_2}
	Step \ref{step:second_ls} of Algorithm \ref{alg::improved_FSD} is well defined if $z_c^k$ is nondominated with respect to points in $\hat{X}^k$.
\end{proposition}
\begin{proof}
	Let $y$ be any point in $\hat{X}^k$; if $F(y)=F(z_c^k)$, then by \cite[Lemma 4]{fliege2000steepest} and the condition $\theta^I(z_c^k)<0$, there exists $\bar{\alpha}>0$ such that $F_I(z_c^k+\alpha v^I(z_c^k))<F_I(z_c^k) = F_I(y)$ for all $\alpha<\bar{\alpha}$; thus there exists $h$ sufficiently large such that $f_j(z_c^k+\alpha_0\delta^h v^I(z_c^k)) < f_j(y)$ for all $j\in I$.
	If, on the other hand, there exists $j \in \left\{1,\ldots, m\right\}$ such that $f_j(z_c^k)<f_j(y)$, then by the continuity of $F$ there exists $\alpha=\alpha_0\delta^h$ sufficiently small such that $f_j(z_c^k+\alpha v^I(z_c^k))<f_j(y)$. Thus, the condition can be satisfied for all $y\in\hat{X}^k$ and  $\alpha_c^I$ is the minimum of the corresponding values of $\alpha_0\delta^h$.
\end{proof}

\begin{proposition}
	\label{prop:nondom}
	If $X^k$ contains mutually nondominated points with respect to $F$, then $\hat{X}^k$ contains nondominated points at any time during iteration $k$; thus step \ref{step:second_ls} is always well defined and $X^{k+1}$ is finally a set of nondominated solutions. 
\end{proposition}
\begin{proof}
	At iteration $k$, the set $\hat{X}^k$ is initialized with the nondominated points $X^k$; then, it is only updated at steps \ref{step:add_zk} and \ref{step:insert_partial}. At step \ref{step:add_zk}, either $z_c^k=x_c$, and the set is not modified, or, by the definition of $\alpha_c^k$, $z_c^k$ dominates $x_c$,  which in turn was nondominated. Thus, the added point $z_c^k$ is nondominated, while all the newly dominated points are removed. At step \ref{step:insert_partial}, the added point $z_c^k+\alpha_c^Iv^I(z_c^k)$ is nondominated by the definition of $\alpha_c^I$; all the newly dominated points are removed. Thus, $\hat{X}^k$ always contains mutually nondominated solutions. By Proposition \ref{prop:well_def_2} step \ref{step:second_ls} is therefore always well defined. Moreover, since $X^{k+1}=\hat{X}^k$ at the end of the iteration, $X^{k+1}$ inherits the nondominance property from $\hat{X}^k$.
\end{proof}

\begin{lemma}
	\label{lemma:lemma}
	After step \ref{step:add_zk} of Algorithm \ref{alg::improved_FSD}, $z^k_c$ belongs to $\hat{X}^k$. Moreover, for all $\tilde{k}>k$, there exists $y\in X^{\tilde{k}}$ such that $F(y)\le F(z_c^k)$.
\end{lemma}
\begin{proof}
	The first assertion of the proposition trivially follows from the update rule of $\hat{X}^k$, at step \ref{step:add_zk}. Now, either $z^{k}_c\in X^{\tilde{k}}$ or $z_c^k\notin X^{\tilde{k}}$; in the former case, we trivially have $y=z_c^k$; otherwise, we can notice that, by the instructions of the algorithm, any set $X^{\tilde{k}}$, $\tilde{k}>k$, is the result of repeated application of steps \ref{step:add_zk} and \ref{step:insert_partial}, starting from $\hat{X}^k$ at some point when $z_c^k\in \hat{X}^k$. When $z_c^k$ was removed from the set, a point $y^1$ was certainly inserted such that $F(y^1)\le F(z_c^k)$. Then, either $y^1\in X^{\tilde{k}}$, or $y^1$ was removed when a point $y^2$ such that $F(y^2)\le F(y^1)$ was added. By recursively applying the reasoning, we have that there is certainly a point $y^t\in X^{\tilde{k}}$ such that $F(y^t)\le F(y^{t-1})\le\ldots \le F(y^2)\le F(y^1)\le F(z_c^k)$. This completes the proof. 
\end{proof}

\begin{proposition}
	Let $X^0$ be a set of mutually nondominated points and $x_0\in X^0$ be a point such that the set $\mathcal{L}(x_0) = \bigcup_{j=1}^{m}\{x\in\mathbb{R}^n\mid f_j(x)\le f_j(x_0)\}$ is compact.
	Let $\{X^k\}$ be the sequence of sets of nondominated points produced by Algorithm \ref{alg::improved_FSD}. Let $\{x_{j_k}\}$ be a linked sequence, then it admits limit points and every limit point is Pareto-stationary for problem \eqref{eq:mo_prob}.
\end{proposition}
\begin{proof}
	For any $k$, either $x_0\in X^k$ or $x_0\notin X^k$. In the former case, since all points in $X^k$ are mutually nondominated, we certainly have $x_{j_k}\in \mathcal{L}(x_0)$. Otherwise, by a similar reasoning as in the proof of Lemma \ref{lemma:lemma}, we have that there is a point $y_k\in X^k$ such that $F(y_k)\le F(x_0)$; since $y_k$ does not dominate $x_{j_k}$, we have that there exists $h \in \left\{1,\ldots, m\right\}$ such that $f_h(x_{j_k})\le f_h(y_k)\le f_h(x_0)$; thus, again, $x_{j_k}\in\mathcal{L}(x_0)$. Therefore the entire sequence $\{x_{j_k}\}$ belongs to the compact set $\mathcal{L}(x_0)$, and thus admits limit points. 
	
	Now, let us consider a limit point $\bar{x}$ of a linked sequence $\{x_{j_k}\}$, i.e., there exists $K\subseteq\{1,2,\ldots\}$ such that
	$$\lim_{\substack{k\to\infty\\k\in K}}x_{j_k} = \bar{x}.$$
	We assume by contradiction that $\theta(\bar{x})<0$ and thus there exists $\varepsilon>0$ such that for all $k\in K$ sufficiently large we have $\theta(x_{j_k})\le -\varepsilon<0$.
	Let $z_{j_k} = x_{j_k}+\alpha_{j_k}v(x_{j_k})$ the point obtained at step \ref{step:z} of the algorithm starting from $x_{j_k}$.
	Now, $\alpha_{j_k}\in[0,\alpha_0]$, which is a compact set, thus there exists a further subsequence $K_1\subseteq K$ such that $\alpha_{j_k}\to \bar{\alpha}\in[0,\alpha_0]$. Moreover, function $v(\cdot)$ is continuous, thus $v(x_{j_k})\to v(\bar{x})$ for $k\to \infty$, $k\in K_1$. Hence, taking the limits along $K_1$ we also get that $z_{j_k}\to \bar{x} + \bar{\alpha}v(\bar{x}) = \bar{z}$.
	
	By the definition of $\alpha_{j_k}$ and $z_{j_k}$ (steps \ref{step:line_search}-\ref{step:z}) we have that 
	$$F(z_{j_k})\le F(x_{j_k}) + \boldsymbol{1}\gamma \alpha_{j_k} \theta(x_{j_k}).$$
	Taking the limits for $k\in K_1$, $k\to \infty$, recalling the continuity of $\theta(\cdot)$, we get
	\begin{equation}
		\label{eq:main_1}
		F(\bar{z}) \le F(\bar{x}) + \boldsymbol{1}\gamma \bar{\alpha} \theta(\bar{x})\le F(\bar{x})-\boldsymbol{1}\gamma \bar{\alpha}\varepsilon.
	\end{equation}
	Now, given $k\in K_1$, let $k_1(k)$ be the smallest index in $K_1$ such that $k_1(k)>k$. By Lemma \ref{lemma:lemma}, there exists $y_{j_{k_1(k)}}\in X^{k_1(k)}$ such that $F(y_{j_{k_1(k)}})\le F(z_{j_k})$; moreover, $x_{j_{k_1(k)}}\in X^{k_1(k)}$; by Proposition \ref{prop:nondom}, the points in $X^{k_1(k)}$ are mutually nondominated, hence there exists $h(k)\in\{1,\ldots,m\}$ such that
	$$f_{h(k)}(x_{j_{k_1(k)}})\le f_{h(k)}(y_{j_{k_1(k)}})\le f_{h(k)}(z_{j_k}).$$
	Considering a further subsequence $K_2\subseteq K_1$ such that $h(k) = h$ for all $k\in K_2$ and taking the limits, we obtain
	$$f_h(\bar{x})\le f_h(\bar{z}).$$
	Putting this last result together with \eqref{eq:main_1}, we get
	$$f_h(\bar{x})\le f_h(\bar{z})\le f_h(\bar{x})-\gamma\bar{\alpha}\varepsilon.$$
	Since $\bar{\alpha}\in[0,\alpha_0]$, $\varepsilon>0$ and $\gamma>0$, the above chain of inequalities can only hold if $\bar{\alpha}=\lim_{k\to \infty,k\in K_2}\alpha_{j_k}=0$. For all $k\in K_2$ sufficiently large, we have $\theta(x_{j_k})<0$ and, thus, $\alpha_{j_k}$ is defined at step \ref{step:line_search}. Since $\alpha_{j_k}\to 0$, for any $q\in\mathbb{N}$, for all $k\in K_2$ large enough we certainly have $\alpha_{j_k}< \alpha_0\delta^q$; thus, the Armijo condition $ F(x_{j_k}+\alpha v(x_{j_k}))\le F(x_{j_k})+\boldsymbol{1}\gamma \alpha \theta(x_{j_k})$ is not satisfied by $\alpha= \alpha_0\delta^q$, i.e., there exists $\tilde{h}(k)$ such that 
	$$f_{\tilde{h}(k)}(x_{j_k}+\alpha_0\delta^q v(x_{j_k})) > f_{\tilde{h}(k)}(x_{j_k})+\gamma \alpha_0\delta^q \theta(x_{j_k}).$$
	Taking the limits along a suitable subsequence such that $\tilde{h}(k)= \tilde{h}$, we get
	$$f_{\tilde{h}}(\bar{x}+\alpha_0\delta^qv(\bar{x}))\ge f_{\tilde{h}}(\bar{x})+\gamma \alpha_0\delta^q\theta(\bar{x}).$$
	Now, since $q$ is arbitrary and $\theta(\bar{x})<0$, this is absurd by \cite[Lemma 4]{fliege2000steepest}. The proof is thus complete.    
\end{proof}

\begin{figure*}[th]
	\centering
	\subfloat[\label{fig::FIG_IFSD-a}]{\includegraphics[width=0.35\textwidth]{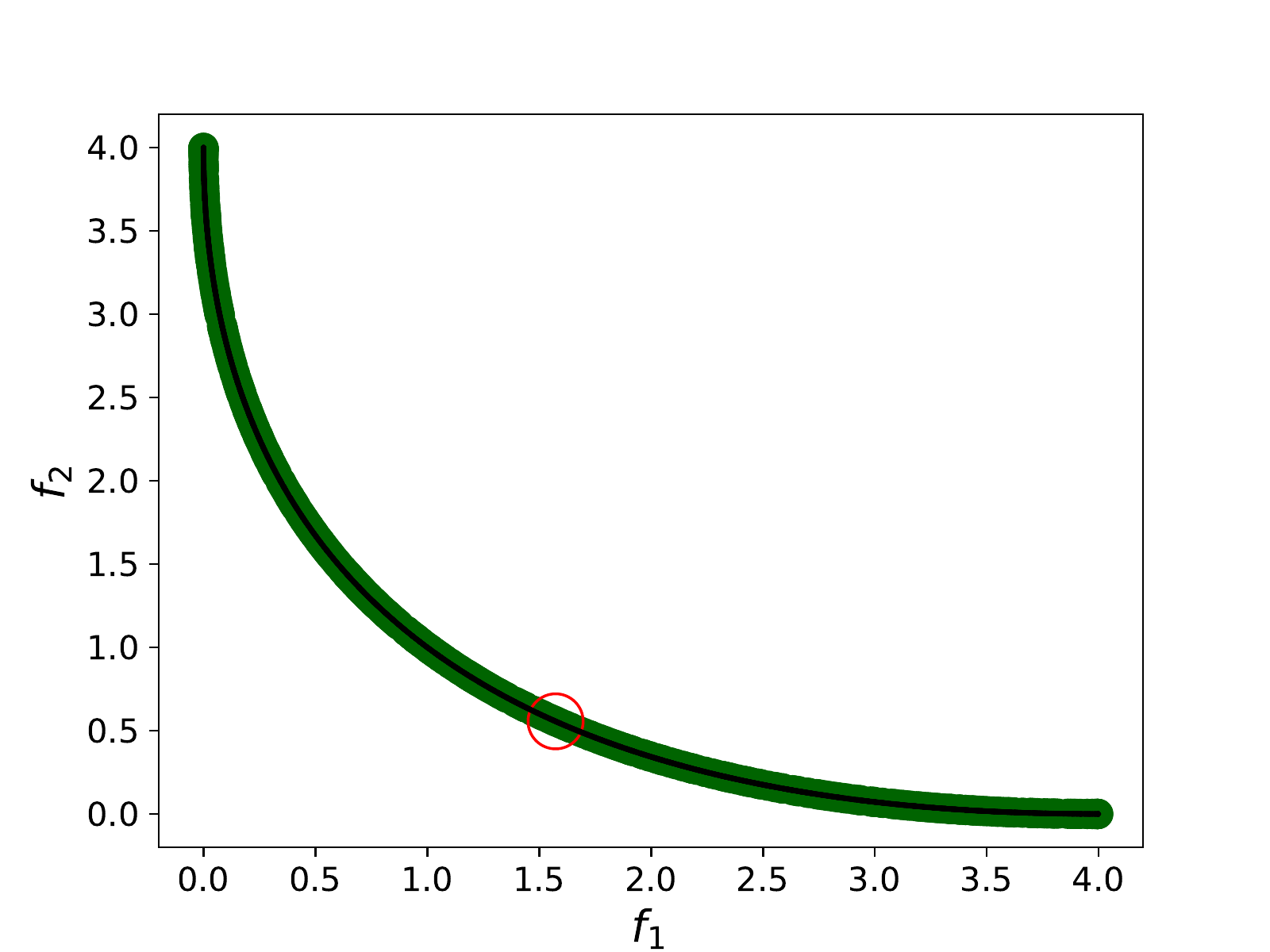}}
	\hfil
	\subfloat[\label{fig::FIG_IFSD-b}]{\includegraphics[width=0.35\textwidth]{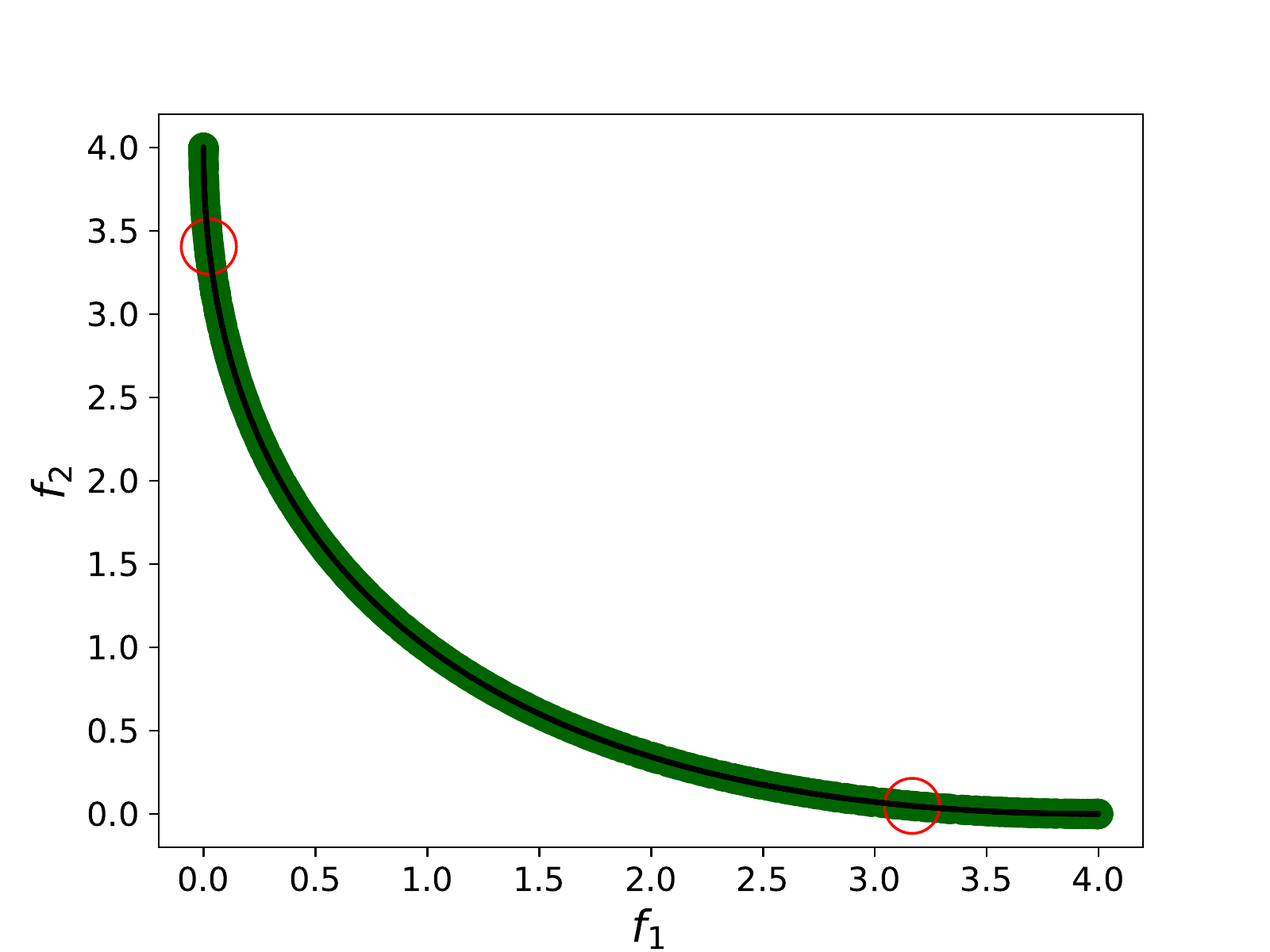}}
	\hfil
	\subfloat[\label{fig::FIG_IFSD-c}]{\includegraphics[width=0.35\textwidth]{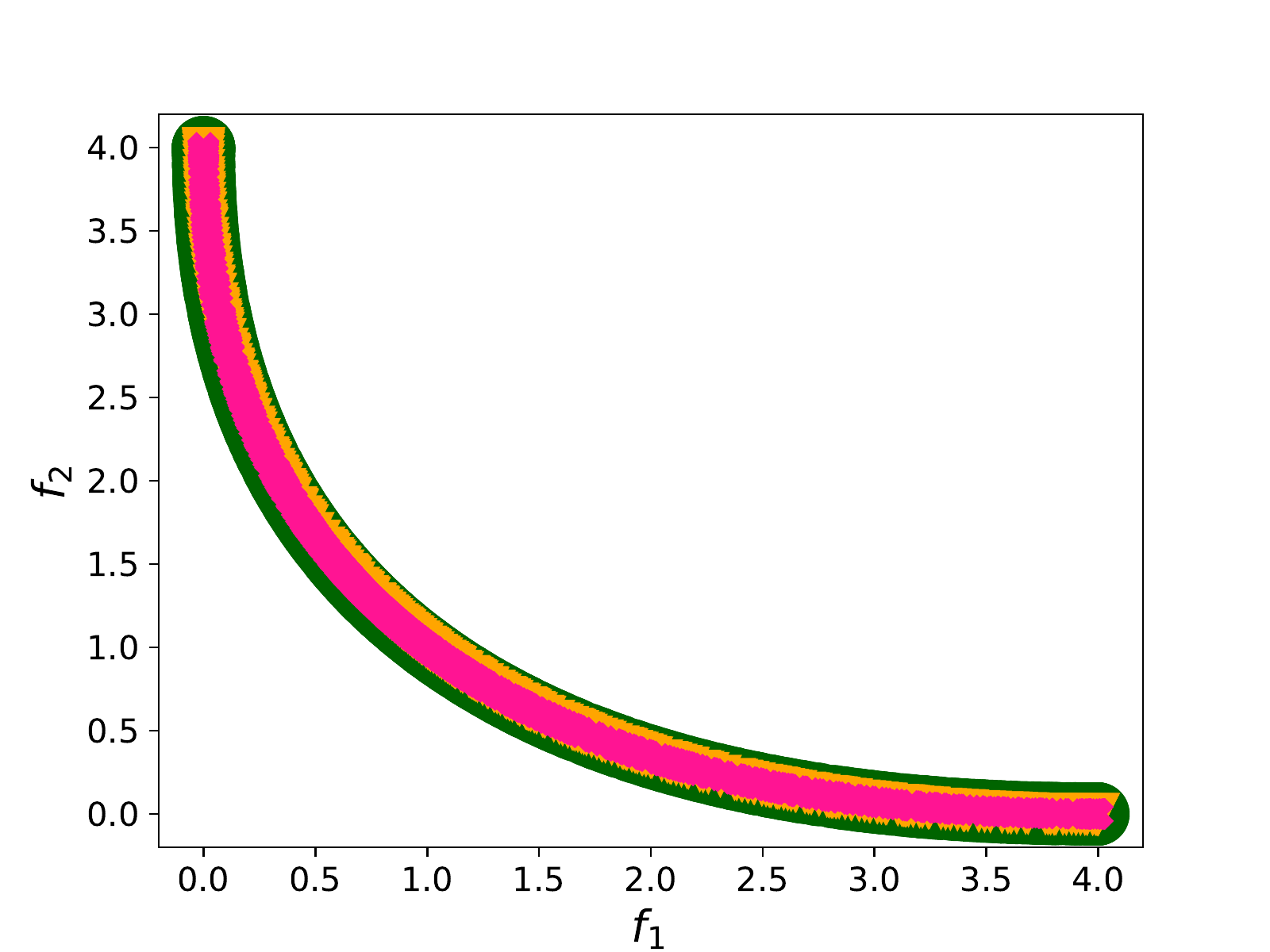}}
	\caption{Pareto fronts obtained by the \texttt{IFSD} algorithm on the convex JOS problem ($n=5$) starting from different initial points: (a) 1 Pareto point as in Figure \ref{fig::FIG_FSDA-a}; (b) 2 Pareto points as in Figure \ref{fig::FIG_FSDA-b}; (c) 3 independent runs from the same random points as those of Figure \ref{fig::FIG_FSDA}(c)-(d).}
	\label{fig::FIG_IFSD}
\end{figure*}

\begin{figure*}
	\centering
	\subfloat[\label{fig:PP-purity} Purity profile]{\includegraphics[width=0.3\textwidth]{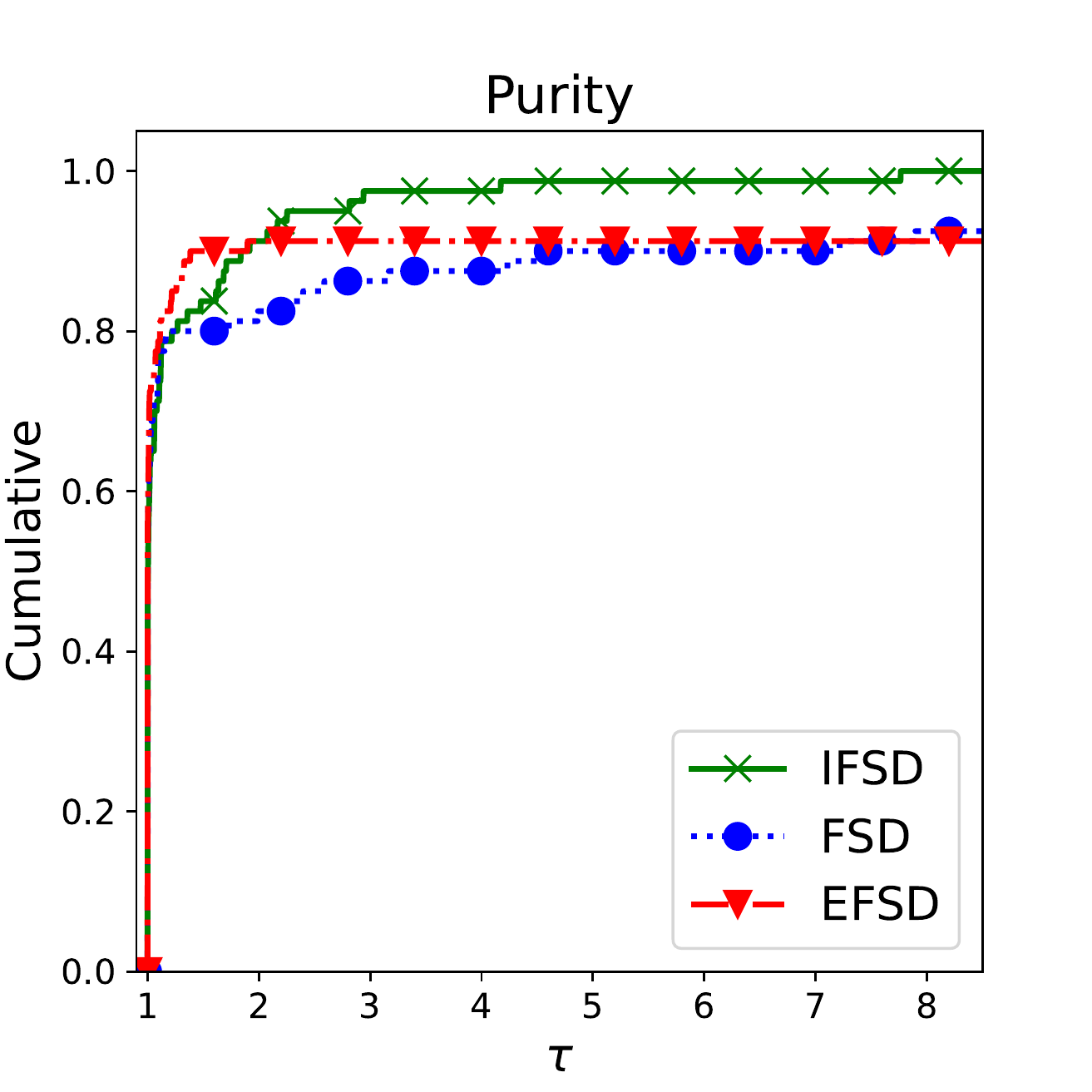}}
	\hfil
	\subfloat[\label{fig:PP-hv} Hypervolume profile]{\includegraphics[width=0.3\textwidth]{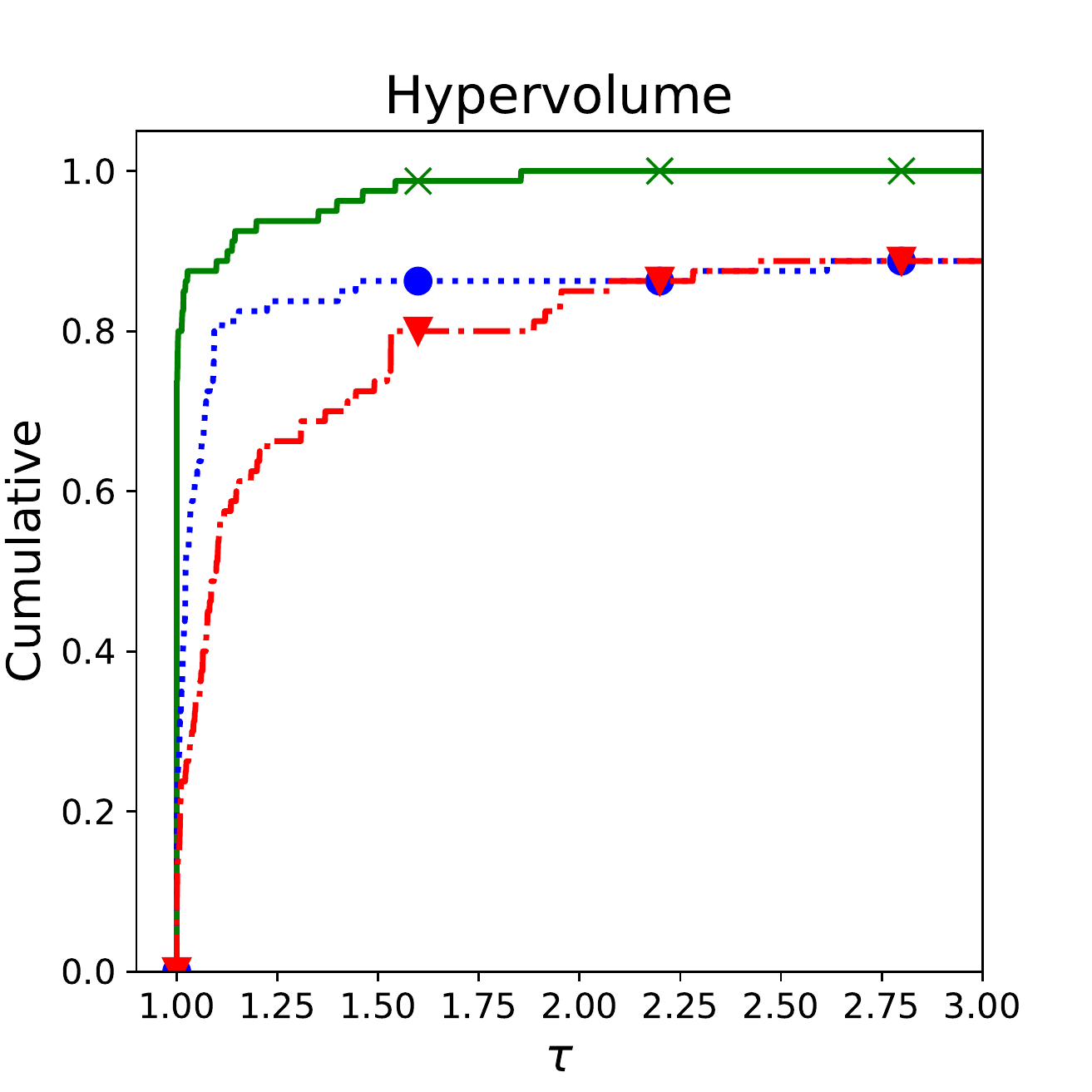}}
	\\
	\subfloat[\label{fig:PP-gamma} $\Gamma$-spread profile ]{\includegraphics[width=0.3\textwidth]{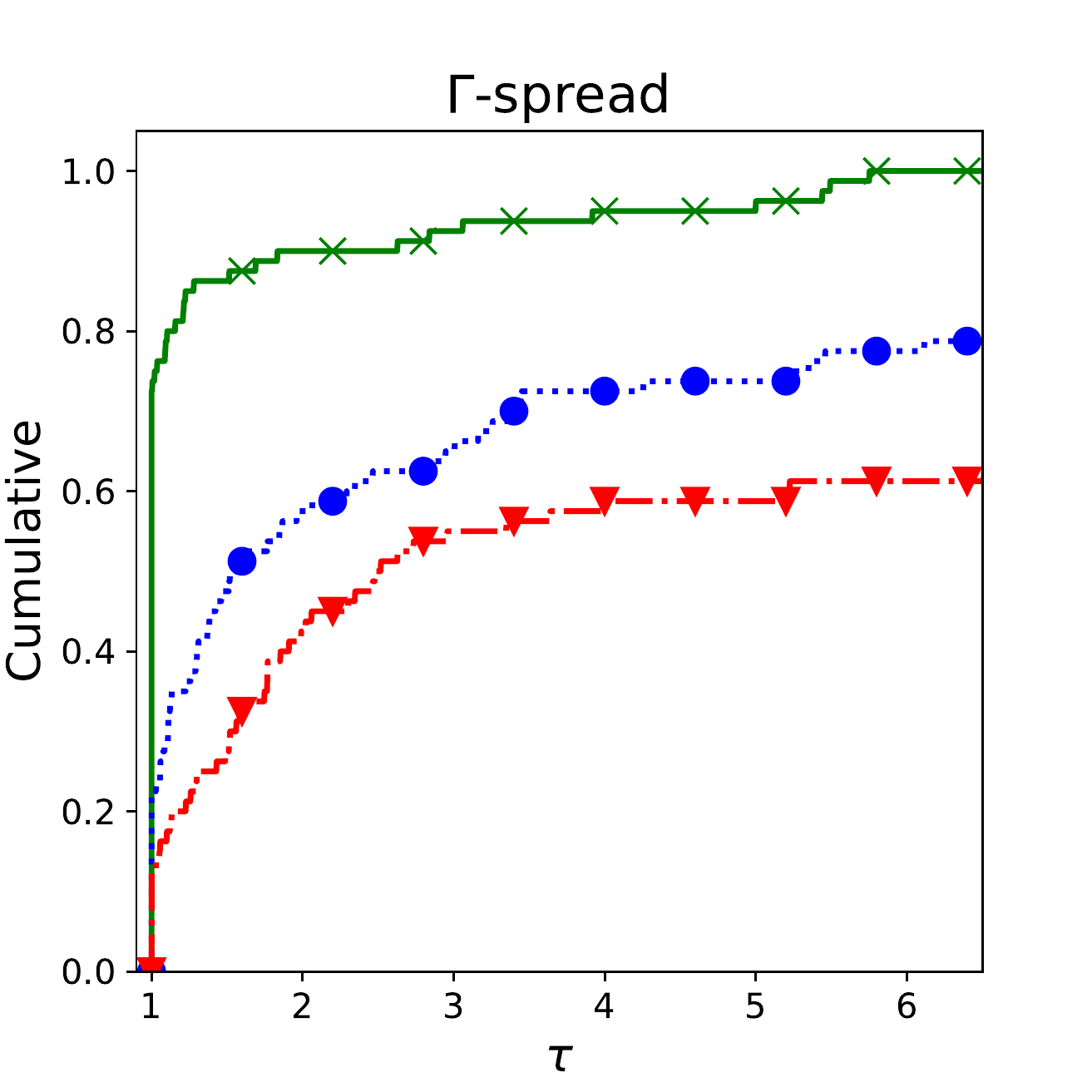}}
	\hfil
	\subfloat[\label{fig:PP-delta} $\Delta$-spread profile]{\includegraphics[width=0.3\textwidth]{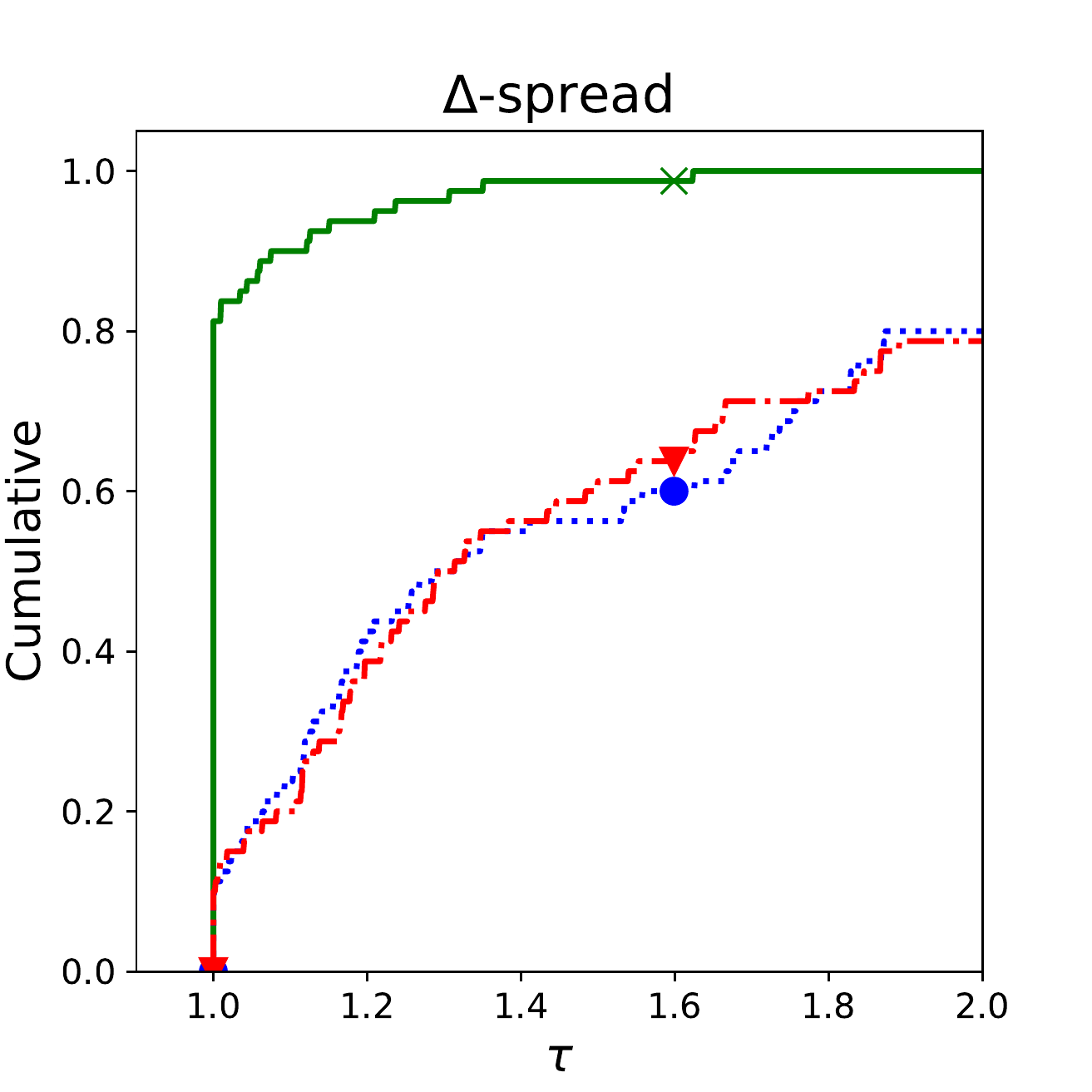}}
	\caption{Performance profiles for the \texttt{IFSD}, \texttt{FSD} and \texttt{EFSD} algorithms on a benchmark of 80 multi-objective problems.}
	\label{fig::PP}
\end{figure*}

\section{Numerical results}
\label{sec:exp}

In this section, we show the results of computational experiments, supporting the discussion in Sections \ref{sec:steepest}-\ref{sec:improved_fsd}. The code, which was written in Python3, was executed on a computer with the following characteristics: Ubuntu 22.04, Intel Xeon Processor E5-2430 v2 6 cores 2.50 GHz, 16 GB RAM. In order to solve instances of problems \eqref{eq::steepest_common}-\eqref{eq::steepest_partial}, the Gurobi optimizer (version 9.5) was employed.

We compared our approach (\texttt{IFSD}) to the original \texttt{FSD}, Algorithm \ref{alg::FSD},  equipped with the base line search (Algorithm \ref{alg::als}) or the extrapolation strategy (\texttt{EFSD}). The following parameters setting was used for line searches: $\alpha_0 = 1, \delta=0.5, \gamma = 10^{-4}$. 

With respect to the conceptual scheme in Algorithm \ref{alg::improved_FSD}, we employed within \texttt{IFSD} a strategy to limit the number of points used for partial descent searches, in order to improve the efficiency of the overall procedure and avoid the production of too many, very close solutions. In particular, we added a condition based on the crowding distance \cite{deb2002fast} to decide whether a point should be considered for further exploration after the steepest descent step or not.  

The benchmark used for the comparisons consists of the following unconstrained problems: CEC09\_2, CEC09\_3 \cite{zhang_multiobjective_2009}, JOS\_1 \cite{jin2001dynamic}, MAN \cite{lapucci2022memetic} ($m=2$) and CEC09\_10 ($m=3$) \cite{zhang_multiobjective_2009}. 
For all the problems, we considered instances with values of $n$ in $\{5,10,20,30,40,50,100,200\}$. Moreover, each problem was tested twice, with different strategies for the initial points: a) $n$ points are uniformly sampled from the hyper-diagonal of a suitable box; b) only the midpoint of the hyper-diagonal is selected. The hyper-diagonal refers to the box constituting the constraints in the bounded version of CEC and MAN problems, whereas it is $[-100, 100]^n$ for the JOS problem.

In order to appreciate the relative performance and robustness of the approaches, we employed the \textit{performance profiles} \cite{dolan2002benchmarking}. In brief, this tool shows the probability that a metric value achieved by a method in a problem is within a factor $\tau \in \mathbb{R}$ of the best value obtained by any of the algorithms in that problem. We employed classical metrics for multi-objective optimization: \textit{purity}, \textit{$\Gamma$--spread}, \textit{$\Delta$--spread} \cite{custodio11} and \textit{hyper-volume} \cite{zitzler98}. 
Purity and hyper-volume have increasing values for better solutions: then, the corresponding profiles are produced considering the inverse of the obtained values.

In Figure \ref{fig::FIG_IFSD}, the behavior of the proposed approach in the same setting as in Figure \ref{fig::FIG_FSDA} is shown. In this example we can observe that now,  regardless, of the starting point(s), the entire Pareto front is effectively spanned, with not even tiny holes.  

For a more consistent assessment of algorithms performance, we report in Figure \ref{fig::PP} the performance profiles for the \texttt{IFSD}, \texttt{FSD} and \texttt{EFSD} algorithms on the entire benchmark of 80 problem instances. 

We observe a remarkable superiority of the proposed approach w.r.t.\ the original variants of the algorithm, especially in terms of the spread metrics, which points out that the Pareto front is indeed spanned more widely and uniformly. The strong hypervolume performance also supports this result. As for purity metric, the three algorithms appear to be closer, but we still observe a slight advantage of \texttt{IFSD}.

\section{Conclusions}
\label{sec:conclusions}
In this paper, we introduced an improved Front Steepest Descent algorithm with asymptotic convergence guarantees similar as those of the original method. The novel algorithm is designed so as to overcome some empirically evident limitation of \texttt{FSD}, that is often unable to span large portions of the Pareto front. Numerical evidence suggests that the proposed procedure effectively achieves this goal.

Future work should be focused on the integration of the proposed approach and the extrapolation strategy proposed in \cite{cocchi2020convergence}. Moreover, the employment of the proposed approach within memetic procedures for global multi-objective optimization \cite{lapucci2022memetic} might be considered. Finally, the algorithm defined in this work could be extended to deal with constrained optimization problems. 
\section*{Conflict of interest}
The authors declare that they have no conflict of interest.

\bibliographystyle{abbrv}
\bibliography{bibliography.bib}

\end{document}